\documentclass[leqno]{article}
\usepackage[left=4cm, right=4cm]{geometry}
\usepackage{amssymb,amsmath,amsthm,enumerate}
\usepackage{secdot}
\usepackage{sectsty}
\sectionfont{\centering \large}
\theoremstyle{theorem}
\newtheorem{theorem}{Theorem}[section]
\newtheorem{lemma}{Lemma}[section]
\newtheorem{corollary}{Corollary}[section]

\theoremstyle{definition}
\newtheorem{definition}{Definition}[section]
\newtheorem{example}{Example}[section]
\numberwithin{equation}{section}
\begin{document}
	\title{A Generalised Continuous Primitive Integral and Some of Its
		Applications}
	\author{S. Mahanta$^{1}$,\quad S. Ray$^{2}$\vspace{2mm}\\
		\em\small ${}^{1,2}$Department of Mathematics, Visva-Bharati, 731235, India.\\
		\em\small ${}^{1}$e-mail: sougatamahanta1@gmail.com\\
		\em\small ${}^{2}$e-mail: subhasis.ray@visva-bharati.ac.in}
	
	\date{}
	\maketitle
	
		\begin{abstract}
			Using the Laplace derivative a Perron type integral, the Laplace integral, is defined. Moreover, it is shown that this integral includes Perron integral and to show that the inclusion is proper, an example of a function is constructed, which is Laplace integrable but not Perron integrable. Properties of integrals such as fundamental theorem of calculus, Hake's theorem, integration by parts, convergence theorems, mean value theorems, the integral remainder form of Taylor's theorem with an estimation of the remainder, are established. It turns out that concerning the Alexiewicz's norm, the space of all Laplace integrable functions is incomplete and contains the set of all polynomials densely. Applications are shown to Poisson integral, a system of generalised ordinary differential equations and higher-order generalised ordinary differential equation. 
		\end{abstract}
	    \thanks{\bf Keywords:} Perron integral, Henstock integral, Laplace integral, Laplace derivative, Laplace continuity, Poisson integral, Ordinary differential equation.\\\\
	    \thanks{\bf Mathematics Subject Classification(2020):} 26A39, 26A27, 34A06.
	
	%% \linenumbers
	
	%% main text
	\section{Introduction}
	\par In integral calculus, there are two fundamental concepts of integration, one is the Riemann integral, and the other is Newton integral, which are not comparable. Even the Lebesgue integral does not contain the Newton integral. This phenomenon leads to the problem of defining an integral which can contain both Lebesgue and Newton integrals. Denjoy, Perron and Henstock individually defined three types of integrals, all of which have continuous primitives and are more general than Lebesgue and Newton integrals (\cite{RSG}). However, all of them are equivalent (Chapter 11 of \cite{RSG}). Using the concept of $ACG$ function and approximate derivative another integral was developed, known as the Khintchine integral or wide Denjoy integral, which is more general than the Perron integral and also has continuous primitives. In the literature, there are lots of other integrals with discontinuous primitives, most of which are developed to recover the coefficients of a trigonometric series (\cite{INTRSEB}, \cite{APMCIN}, \cite{AGNAPTRSR}, \cite{ILPIntegral}). As integrals with continuous primitives have smoother applications than integrals with discontinuous primitives, it will be our benefit if we have sufficient number integrals with continuous primitives.
	\par In this paper, using Laplace derivative (\cite{TLD}, \cite{OLD}, \cite{TLD1}, \cite{TLD2}, \cite{OTSymLapDer}), a Perron type integral is defined, which has continuous primitives (Theorem \ref{Continuous primitive}). We call it the Laplace integral. Laplace derivative is also used in \cite{APTILD} to define $L_{n}$-integral for $n\in\mathbb{N}$. It is also possible to define $L_{0}$-integral similarly as $L_{n}$-integral ($n\geqslant 1$), which turns out to be equivalent to the Laplace integral. This paper is organized in the following order. After discussing some preliminary definitions in Section \ref{Preliminaries}, we present the definition of Laplace integral in Section \ref{Definition}. In Section \ref{NeWexample}, a continuous function $f\colon[0,1]\to\mathbb{R}$ is constructed, which is Laplace integrable but not Perron integrable. All of the properties of this integral are given in Section \ref{Basic Properties}--\ref{Taylor-theorem}. Section \ref{Poisson integral} appears with an application of this integral to Poisson integral and its boundary behaviour. In Section \ref{Exist-uniqu-IVP}, existence and uniqueness theorems for a system of generalised ordinary differential equations and higher-order generalised ordinary differential equation are given.

	\section{Preliminaries}\label{Preliminaries}

	\begin{definition}[\cite{OLC}]
		Let $f$ be Perron integrable on a neighbourhood of $x$. Then $f$ is said to be Laplace differentiable at $x$ if $\exists\, \delta>0$ such that the following limits
		\[
		\lim\limits_{s\rightarrow \infty}s^{2}\int_{0}^{\delta}e^{-st}[f(x+t)-f(x)]\,dt
		\]
		and
		\[
		\lim\limits_{s\rightarrow \infty}(-s^{2})\int_{0}^{\delta}e^{-st}[f(x-t)-f(x)]\,dt
		\]
		exist and are equal. In this case, the common value is called the Laplace derivative of $f$ at $x$ and denoted by $LD_{1}f(x)$.
	\end{definition}

	\par Definitions of $\underline{LD}_{1}^{+}f(x)$, $\underline{LD}_{1}^{-}f(x)$, $\overline{LD}_{1}^{+}f(x)$, $\overline{LD}_{1}^{\,-}f(x)$ are obvious. For the definition of higher order Laplace derivative, $LD_{n}f$ $(n\geqslant 2)$, see \cite{OLD}, \cite{HOD}.

	\begin{lemma}\label{lemma11}
		Let $\delta>0$ and let $F(s,\delta)=s^{2}\int_{0}^{\delta}e^{-st}\phi(t)\,dt$, where $\phi$ is Perron integrable. Then the limits $\liminf\limits_{s\rightarrow \infty} F(s,\delta)$, $\limsup\limits_{s\rightarrow \infty} F(s,\delta)$ do not depend on $\delta$.
	\end{lemma}
	\begin{proof}
		Let $0<\delta_{1}<\delta_{2}$. Then
		\begin{align*}
		s^{2}\int_{0}^{\delta_{2}}e^{-st}\phi(t)\,dt=s^{2}\int_{0}^{\delta_{1}}e^{-st}\phi(t)\,dt + s^{2}\int_{\delta_{1}}^{\delta_{2}}e^{-st}\phi(t)\,dt.
		\end{align*}
		Using integration by parts, it is easy to prove that $\lim\limits_{s\rightarrow \infty}s^{2}\int_{\delta_{1}}^{\delta_{2}}e^{-st}\phi(t)\,dt=0$ and this completes the proof.
	\end{proof}

	\begin{theorem}
		Let $f$ be Perron integrable on $[a,b]$. Then the definitions of $\underline{LD}_{1}^{-}f$, $\underline{LD}_{1}^{+}f$, $\overline{LD}_{1}^{\,-}f$ and $\overline{LD}_{1}^{+}f$ do not depend on $\delta$.
	\end{theorem}
	\begin{proof}
		Consider $\phi(t)=f(x+t)-f(x)$. Now using Lemma \ref{lemma11}, we will arrive at our conclusion.
	\end{proof}

	\section{Definition of Laplace integral}\label{Definition}
	\begin{definition}
		Let $f\colon[a,b]\to \mathbb{R}$ and let $U\colon[a,b]\to \mathbb{R}$. Then $U$ is said to be a major function of $f$ if $U$ is continuous, $\underline{LD}_{1}U \geqslant f$ on $[a,b]$ and $\underline{LD}_{1}U>-\infty$ on $[a,b]$.
	\end{definition}

	\par A function $V\colon[a,b]\to \mathbb{R}$ is said to be a minor function of $f$ if $-V$ is a major function of $f$.

	\begin{lemma}\label{increasinglemma}
		If $U$ and $V$ are a major and a minor function of $f$ respectively, then $U-V$ is non-decreasing.
	\end{lemma}
	\begin{proof}
		As $U-V$ is continuous and $\underline{LD_{1}}(U-V)\geqslant \underline{LD_{1}}U - \overline{LD_{1}}V \geqslant 0$ on $[a,b]$, from Theorem $19$ of \cite{OLD}, we get $U-V$ is non-decreasing.
	\end{proof}
	\par From now on, for any function $F\colon[a,b]\to\mathbb{R}$, we use the symbol $F_{a}^{b}$ to denote the difference $F(b)-F(a)$.

	\begin{definition}
		Let $f\colon[a,b]\to \mathbb{R}$. By Lemma \ref{increasinglemma}, we have
		\begin{align*}
		\sup \left\{ V_{a}^{b}\mid V\, \mbox{is a minor function of }\, f  \right\}\leqslant \inf \left\{ U_{a}^{b}\mid U\, \mbox{is a major function of }\, f  \right\}.
		\end{align*}
		If the equality holds with a finite value, then we say $f$ is Laplace integrable on $[a,b]$. In this case, the common value is said to be the Laplace integral of $f$ on $[a,b]$ and is denoted by $\int_{a}^{b}f$. Moreover, we denote the set of all Laplace integrable functions on $[a,b]$ by $\mathcal{LP}[a,b]$.
	\end{definition}

	\par A continuous function $F\colon[a,b]\to\mathbb{R}$ is said to be a primitive of $f$, where $f\in\mathcal{LP}[a,b]$, if $F(x)-F(a)=\int_{a}^{x}f$ for $x\in [a,b]$.

	\begin{corollary}\label{corollary_1}
		Let $f\colon[a,b]\to\mathbb{R}$. Then $f\in \mathcal{LP}[a,b]$ if and only if for any $\epsilon> 0$ there exist a major function $U$ and a minor function $V$ of $f$ such that
		\[
		U_{a}^{b} - V_{a}^{b}< \epsilon.
		\]
	\end{corollary}

	\par By Theorem 2.18.3 of \cite[p.~167]{HOD} it is evident that $\underline{D}f\leqslant \underline{LD_{1}}f\leqslant \overline{LD_{1}}f\leqslant \overline{D}f$, where $\underline{D}f$ and $\overline{D}f$ are Dini derivates. It proves that Laplace integral includes Perron integral. The converse is not valid (see Section \ref{NeWexample}).

	\begin{definition}\label{exmajor}
		A continuous function $U\colon[a,b]\to \mathbb{R}$ is said to be ex-major function of $f\colon[a,b]\to \mathbb{R}$ if $\underline{LD}_{1}U\geqslant f$ a.e. on $[a,b]$ and $\underline{LD}_{1}U>-\infty$ nearly everywhere on $[a,b].$
	\end{definition}
	
	\par Definition of ex-minor function of $f\colon[a,b]\to \mathbb{R}$ is analogous.

	\begin{definition}\label{exlaplaceintegral}
		A function $f\colon[a,b]\to \mathbb{R}$ is said to be Laplace-ex integrable if 
		\begin{align*}
		-\infty&<\inf \left\{ U_{a}^{b} \,\,|\,\, U\,\, \mbox{is an ex-major function of }\,\, f  \right\}\\
		&=\sup \left\{ V_{a}^{b} \,\,|\,\, V\,\, \mbox{is an ex-minor function of }\,\, f  \right\}<\infty.
		\end{align*}
		Furthermore, the class of all Laplace-ex  integrable functions on $[a,b]$ will be denoted by $\mathcal{LP}_{ex}[a,b]$.
	\end{definition}

	\par With the help of next three Lemmas, it can be proved that $\mathcal{LP}[a,b]=\mathcal{LP}_{ex}[a,b]$. However, we omit the proofs as they are quite similar to the proofs of the corresponding results for Perron integral (Lemma 8.22--8.24 and Theorem 8.25 of \cite{RSG}).

	\begin{lemma}\label{corrlm1}
		Let $W\colon[a,b]\to \mathbb{R}$ be continuous, and $c\in[a,b]$, and let $\epsilon >0.$ Then there exists a non-decreasing continuous function $\psi\colon[a,b]\to \mathbb{R}$ and a $\delta(>0)$ such that $\psi(a)=0,$ $\psi(b)\leqslant \epsilon$ and for all $(s,\delta)\in [0,\infty)\times(0,\infty)$,
		\[
		s^{2}\int_{0}^{\delta}e^{-st}[W(c+t)-W(c)+\psi(c+t)-\psi(c)]\,dt\geqslant 0
		\]
		and
		\[
		(-s^{2})\int_{0}^{\delta}e^{-st}[W(c-t)-W(c)+\psi(c-t)-\psi(c)]\,dt\geqslant 0.
		\]
	\end{lemma}

	\begin{lemma}\label{corrlm2}
		Let $W\colon[a,b]\to \mathbb{R}$ be continuous, let $\epsilon>0,$ and suppose that $\underline{LD}_{1}W>-\infty$ nearly everywhere on $[a,b].$ Then there exists a continuous function $Y\colon[a,b]\to \mathbb{R}$ such that $\underline{LD}_{1}Y\geqslant \underline{LD}_{1}W,$ $\underline{LD}_{1}Y>-\infty$ on $[a,b]$ and $Y^{b}_{a}\leqslant W^{b}_{a}+\epsilon.$
	\end{lemma}

	\begin{lemma}\label{corrlm3}
		Let $f\colon[a,b]\to \mathbb{R}$. If $W$ is ex-major function of $f$ on $[a,b]$ and $\epsilon>0$, then there exists a major function $U$ of $f$ on $[a,b]$ such that $U_{a}^{b}<W_{a}^{b}+\epsilon$.
	\end{lemma}

	\begin{theorem}\label{exequivalent}
		Let $f\colon[a,b]\to \mathbb{R}$. Then $f\in \mathcal{LP}[a,b]$ iff $f\in \mathcal{LP}_{ex}[a,b]$.
	\end{theorem}

	\section{Construction of a Laplace integrable function which is not Perron integrable}\label{NeWexample}

	\par Here we construct a function $f\in\mathcal{LP}[0,1]\setminus \mathsf{P}[0,1]$, where $\mathsf{P}[a,b]$ denotes the set of all Perron integrable functions defined over $[a,b]$. For this, we need to define the Smith-Volterra-Cantor set ($\textup{SVC(4)}$), which is given below.

	\begin{definition}[\bf SVC(4)]\label{Definition1}
		Let $S_{0}=[0,1]$. If
		\begin{align}\label{component intervals}
		S_{n-1}= \bigcup_{k}[a_{k},b_{k}]\qquad\text{for $n\geqslant 1$},
		\end{align}
		define
		\begin{align*}
		S_{n}=\bigcup_{k}\left(\left[a_{k},\frac{a_{k}+b_{k}}{2}-\frac{1}{2^{2n+1}}\right]\cup\left[\frac{a_{k}+b_{k}}{2}+\frac{1}{2^{2n+1}},b_{k}\right]\right).
		\end{align*}
		Then the set $S=\cap_{n\in\mathbb{N}}S_{n}$ is called the Smith-Volterra-Cantor set and denoted by $\textup{SVC(4)}$. 
	\end{definition}

	\par Note that the set $S_{n}$ is created from $S_{n-1}$ by removing intervals of length $1/4^{n}$. If we call the intervals $[a_{k},b_{k}]$ in \eqref{component intervals} the component intervals of $S_{n-1}$, then by a straightforward induction it can be proved that the length of each component interval of $S_{n}$ is $(2^{n}+1)/(2^{2n+1})$ for $n\geqslant 0$. The set $S$ is a Cantor-like set with $\lambda(S)=1/2$, where $\lambda$ stands for the Lebesgue measure.
	\par Let $\mathcal{G}_{n}^{n+k}$ be the set of all intervals removed from $[0,1]$ to construct $S_{n+k}$ from $S_{n}$, where $k$ is a positive integer. Then it is evident that $\mathcal{G}=\cup_{n\in \mathbb{N}}\mathcal{G}_{n}^{n+k}$ is the collection of contiguous intervals of $S$ in $[0,1]$.

	\begin{lemma}\label{Lemma1}
		Let $a$ be an element of $S$. Then there exists a sequence of intervals $\left\{\left(a_{n_{k}},b_{n_{k}}\right)\right\}_{k\geqslant 2}$ in $\mathcal{G}$ such that for all $k\geqslant 2$, $\left(a_{n_{k}},b_{n_{k}}\right)\in \mathcal{G}_{n_{k}}^{2n_{k}}$ and either
		\begin{align}
		\left(a_{n_{k}},b_{n_{k}}\right)&\subseteq \left[a+\frac{1}{4^{n_{k}+1}} , a+\frac{5}{4^{n_{k}+1}}  \right]\label{Lemma1a}
		\end{align}
		or,
		\begin{align}
		\left(a_{n_{k}},b_{n_{k}}\right)&\subseteq \left[a-\frac{5}{4^{n_{k}+1}},a-\frac{1}{4^{n_{k}+1}}  \right].\label{Lemma1b}
		\end{align}
	\end{lemma}
	\begin{proof}
		Note that $a\in S_{n}$ for all $n\in \mathbb{N}$. Moreover, for all $n\geqslant 2$,
		\begin{gather}
		\frac{2^{n}+1}{2^{2n+1}}\geqslant \frac{10}{4^{n+1}} \label{observeqn}
		\end{gather}
		and
		\begin{gather}
		\frac{4}{4^{n+1}}> \frac{2^{2n}+1}{2^{4n+1}}.\label{observeqn1}
		\end{gather}
		Now fix an $n\geqslant 2$. Let $J_{0}$ be the component interval of $S_{n}$ containing $a$. Then $\left[a,a+5/4^{n+1}\right]\subseteq S_{n}$ or, $\left[a-5/4^{n+1},a \right]\subseteq S_{n}$; otherwise, we will have 
		$$J_{0}\subset \left[a-5/4^{n+1},a+5/4^{n+1}\right],$$
		which contradicts \eqref{observeqn}. Suppose $\left[a,a+5/4^{n+1}\right]\subseteq J_{0}\subseteq S_{n}$.
		\par Let $I_{0}=\left[a+1/4^{n+1}, a+5/4^{n+1}\right]$. Then $I_{0}\subseteq J_{0}$ and the length of $I_{0}$, say $l(I_{0})$, is greater than the length of the component intervals of $S_{2n}$ (see \eqref{observeqn1}), which implies $I_{0}$ cannot be a subset of $S_{2n}$.
		\par Assume no member of $\mathcal{G}_{n}^{2n}$ is a subset of $I_{0}$. If $K_{0}$ is the open interval of length $1/4^{n+1}$, which is removed from $J_{0}$ to construct $S_{n+1}$ from $S_{n}$, then it will divide $J_{0}$  into two closed disjoint sub-intervals, say $J_{0,1}$ and $J_{0,2}$. Each of which is of length $(2^{n+1}+1)/2^{2n+3}$. According to our assumption $K_{0}\nsubseteq I_{0}$. Thus $I_{1}=I_{0}\setminus K_{0}$ is connected and either $I_{1}\subseteq J_{0,1}$ or $I_{1}\subseteq J_{0,2}$. Without any loss of generality assume $I_{1}\subseteq J_{0,2}$ and set $J_{1}=J_{0,2}$. Then $J_{1}$ is a component interval of $S_{n+1}$ such that $I_{1}\subseteq J_{1}$ and $I_{1}\subseteq I_{0}$. Having chosen $I_{0},I_{1},...,I_{n-1}$ and $J_{0},J_{1},...,J_{n-1}$ in the process of constructing $S_{2n-1}$ step-by-step from $S_{n}$, suppose $K_{n-1}$ be the open interval of length $1/4^{2n}$, which is removed from $J_{n-1}$ to construct $S_{2n}$ from $S_{2n-1}$. Then it will divide $J_{n-1}$ into two disjoint sub-intervals, say $J_{n-1,1}$ and $J_{n-1,2}$, each of which is of length $(2^{2n}+1)/2^{4n+1}$. Again by our assumption $K_{n-1}\nsubseteq I_{n-1}$. Thus $I_{n}=I_{n-1}\setminus K_{n-1}$ is connected and either $I_{n}\subseteq J_{n-1,1}$ or $I_{n}\subseteq J_{n-1,2}$. Suppose $I_{n}\subseteq J_{n-1,2}$ and set $J_{n}=J_{n-1,2}$. Then $J_{n}$ is a component interval of $S_{2n}$ such that $I_{n}\subseteq J_{n}$ and $I_{n}\subseteq I_{n-1}$. Thus
		\begin{eqnarray}\label{Eq1}
		l(I_{n})\leqslant l(J_{n}).
		\end{eqnarray}
		By the construction, we get $I_{n}=I_{0}\setminus\left(K_{0}\cup K_{1}\cup ...\cup K_{n-1}\right)$ and $K_{i}\cap K_{j}=\phi$ for $i\neq j$. Which implies
		\begin{align*}\label{Eq2}
		l(I_{n})&\geqslant l(I_{0})-\sum_{i=0}^{n-1}l(K_{i})\\
		&= \frac{1}{4^{n}}-\sum_{i=1}^{n}\frac{1}{4^{n+i}}= \frac{2^{2n+1}+1}{3\times 2^{4n}}>\frac{2^{2n}+1}{2^{4n+1}}=l(J_{n}),
		\end{align*}
		contradicting \eqref{Eq1}. Thus there is an interval $\left(a_{n},b_{n}\right)\in \mathcal{G}_{n}^{2n}$, which is a subset of $I_{0}$. Similarly assuming $\left[a-5/4^{n+1},a \right]\subseteq S_{n}$,
		we can prove that there exists an interval	$\left(a_{n},b_{n}\right)\in \mathcal{G}_{n}^{2n}$, which is a subset of $\left[a-5/4^{n+1},a-1/4^{n+1}\right]$. Now define,
		\begin{align*}
		&P^{+}_{a}=\left\{ n\in\mathbb{N}\mid n\geqslant 2\,\,\text{and}\,\,\left[a,a+\frac{5}{4^{n+1}}\right]\subseteq S_{n} \right\},\\
		&P^{-}_{a}=\left\{ n\in\mathbb{N}\mid n\geqslant 2\,\,\text{and}\,\,\left[a-\frac{5}{4^{n+1}},a \right]\subseteq S_{n} \right\}.
		\end{align*}
		Then at least one of them will be infinite, which completes the proof.
	\end{proof}

	\begin{lemma}\label{Corollary1}
		Let $a$ be an element of $S$ and let $\left\{\left(a_{n_{k}},b_{n_{k}}\right)\right\}_{k\geqslant 2}$ be a sequence in $\mathcal{G}$ such that $\left(a_{n_{k}},b_{n_{k}}\right)\in \mathcal{G}_{n_{k}}^{2n_{k}}$.
		\begin{enumerate}[\upshape A.]
			\item If $\left\{\left(a_{n_{k}},b_{n_{k}}\right)\right\}_{k\geqslant 2}$ satisfies \eqref{Lemma1a}, then $\exists$ a sequence of pairs $\left\{\left\{ u_{n_{k}},v_{n_{k}}\right\} \right\}_{k\geqslant 2}$ such that
			\begin{gather*}
			\left\{u_{n_{k}},v_{n_{k}}\right\}\subseteq \left(d_{n_{k}},c_{n_{k}}\right)\subseteq \left(a_{n_{k}},b_{n_{k}}\right),\\
			\left(u_{n_{k}}-a_{n_{k}}\right)^{-\frac{7}{4}}= \left(4l(k)+1\right)\frac{\pi}{2}\quad\text{and}\quad \left(v_{n_{k}}-a_{n_{k}}\right)^{-\frac{7}{4}}= 2l(k)\pi,
			\end{gather*} 
			where $c_{n_{k}}=(a_{n_{k}}+b_{n_{k}})/2\,$, $d_{n_{k}}=(a_{n_{k}}+c_{n_{k}})/2$ and $l(k)$ is an integer depending on $k$. \label{Corollary1a}
			\item If $\left\{\left(a_{n_{k}},b_{n_{k}}\right)\right\}_{k\geqslant 2}$ satisfies \eqref{Lemma1b}, then $\exists$ a sequence of pairs $\left\{\left\{ u_{n_{k}},v_{n_{k}}\right\} \right\}_{k\geqslant 2}$ such that
			\begin{gather*}
			\left\{u_{n_{k}},v_{n_{k}}\right\}\subseteq \left(c_{n_{k}},d_{n_{k}}\right)\subseteq \left(a_{n_{k}},b_{n_{k}}\right),\\
			\left(b_{n_{k}}-u_{n_{k}}\right)^{-\frac{7}{4}}= \left(4m(k)+1\right)\frac{\pi}{2}\quad\text{and}\quad\left(b_{n_{k}}-v_{n_{k}}\right)^{-\frac{7}{4}}= 2m(k)\pi,
			\end{gather*}
			where $c_{n_{k}}=(a_{n_{k}}+b_{n_{k}})/2\,$, $d_{n_{k}}=(b_{n_{k}}+c_{n_{k}})/2$ and $m(k)$ is an integer depending on $k$. \label{Corollary1b}
		\end{enumerate}
	\end{lemma}
	\begin{proof}
		We will give proof of the first part only. The second part is similar. Let $\left(a_{n},b_{n}\right)\in \mathcal{G}_{n}^{2n}$ and let $\left(a_{n},b_{n}\right)\subseteq \left[a+1/4^{n+1}, a+5/4^{n+1}\right]$. Then the length of $\left(a_{n},b_{n}\right)$ is of the form $1/4^{m(n)}$ where $m(n)\in\{n+1,\,...\,,2n\}$ for $n\in \mathbb{N}$. Thus
		\begin{eqnarray}\label{nondifferentiableinequality}
		\frac{1}{4^{2n}}\leqslant b_{n}-a_{n}=\frac{1}{4^{m(n)}}\leqslant \frac{1}{4^{n+1}}.
		\end{eqnarray}
		Define $f(x)=\left(x-a_{n}\right)^{-\frac{7}{4}}$ for $x\in \left(a_{n},b_{n}\right)$. If $c_{n}=(a_{n}+b_{n})/2$ and $d_{n}=(a_{n}+c_{n})/2\,$, then for $n\geqslant 2$,
		\begin{align}
		\begin{split}
		f\left(d_{n}\right)-f\left(c_{n}\right)&=\left(d_{n}-a_{n}\right)^{-\frac{7}{4}}-\left(c_{n}-a_{n}\right)^{-\frac{7}{4}}=2^{\frac{7(2m(n)+2)}{4}}-2^{\frac{7(2m(n)+1)}{4}}\\
		&= 2^{\frac{7(m(n)+1)}{2}}\left(1-2^{-\frac{7}{4}}\right)>2^{\frac{7(n+1)}{2}}\left(1-2^{-1}\right)=2^{\frac{7n+5}{2}}>4\pi.\label{Eq4}
		\end{split}
		\end{align}
		By the intermediate value theorem and \eqref{Eq4}, we will get a pair of real numbers $u_{n}, v_{n}$ such that $\left\{u_{n}, v_{n}\right\}\subseteq\left(d_{n},c_{n}\right)$ and
		\begin{align*}
		&f\left( u_{n} \right)=\left(u_{n}-a_{n}\right)^{-\frac{7}{4}}= \left(4l(n)+1\right)\frac{\pi}{2},\\
		&f\left( v_{n} \right)=\left(v_{n}-a_{n}\right)^{-\frac{7}{4}}= 2l(n)\pi,\,\,\text{for}\,\,n\geqslant 2,
		\end{align*}
		where $l(n)$ is an integer depending on $n$. This completes the proof.
	\end{proof}

	\par Let $\left\{\left(a_{n},b_{n}\right)\right\}$ be an enumeration of $\mathcal{G}$, let $c_{n}=(a_{n}+b_{n})/2$ and let $d_{n}=(a_{n}+c_{n})/2$ for all $n\in \mathbb{N}$. Define
	\begin{gather}
	g_{n}(x)=
	\begin{cases}
	\left(x-a_{n}\right)^{\frac{1}{4}}\sin\left(x-a_{n}\right)^{-\frac{7}{4}}&\text{if $a_{n}<x\leqslant c_{n}$}\\
	\qquad\qquad 0&\text{if $x=a_{n}$},
	\end{cases}\\
	h_{n}(x)=
	\begin{cases}
	\left(b_{n}-x\right)^{\frac{1}{4}}\sin\left(b_{n}-x\right)^{-\frac{7}{4}}&\text{if $c_{n}\leqslant x<b_{n}$}\\
	\qquad\qquad 0&\text{if $x=b_{n}$},
	\end{cases}
	\end{gather}
	and
	\begin{gather}
	f_{n}(x)=
	\begin{cases}
	g_{n}(x)&\mbox{if}\,\,a_{n}\leqslant x\leqslant c_{n}\\
	h_{n}(x)&\mbox{if}\,\,c_{n}\leqslant x\leqslant b_{n}.
	\end{cases}
	\end{gather}
	Then $f_{n}$ is continuous on $[a_{n},b_{n}]$ and differentiable on $(a_{n},b_{n})$ for all $n\in\mathbb{N}$. Now define $f\colon [0,1]\to \mathbb{R}$ by
	\begin{gather}\label{Eq13}
	f(x)=
	\begin{cases}
	f_{n}(x)&\mbox{on}\,\,\left(a_{n},b_{n}\right)\quad \text{for $n\in \mathbb{N}$},\\
	\quad 0&\mbox{on}\,\,S.
	\end{cases}
	\end{gather}
	By the construction, it is clear that $f$ is continuous on $[0,1]$ and differentiable on $[0,1]\setminus S$. Moreover, note that if
	\begin{gather}
	G_{n}(x)=
	\begin{cases}
	(x-a_{n})^{3}\cos(x-a_{n})^{-\frac{7}{4}}&\text{if $a_{n}<x\leqslant c_{n}$}\\
	\qquad\qquad 0&\text{if $x=a_{n}$}
	\end{cases}
	\end{gather}
	and
	\begin{gather}
	H_{n}(x)=
	\begin{cases}
	(b_{n}-x)^{3}\cos(b_{n}-x)^{-\frac{7}{4}}&\text{if $c_{n}\leqslant x<b_{n}$}\\
	\qquad\qquad 0&\text{if $x=b_{n}$}
	\end{cases}
	\end{gather}
	then
	\begin{gather}
	g_{n}(x)=\frac{4}{7}G_{n}^{'}(x)-\frac{12}{7}(x-a_{n})^{2}\cos(x-a_{n})^{-\frac{7}{4}}\quad\mbox{if}\,\,a_{n}<x\leqslant c_{n};\label{Eq7}
	\end{gather}
	and
	\begin{gather}
	h_{n}(x)=-\frac{4}{7}H_{n}^{'}(x)-\frac{12}{7}(b_{n}-x)^{2}\cos(b_{n}-x)^{-\frac{7}{4}}\quad\mbox{if}\,\,c_{n}\leqslant x< b_{n}.\label{Eqnew10}
	\end{gather}

	\begin{theorem}\label{remarkable lemma1}
		Let $f$ be the function defined in \eqref{Eq13}. Then $f$ is not differentiable on $S$.
	\end{theorem}
	\begin{proof}
		Let $a\in S$. Without loss of generality, assume $\left\{\left\{ u_{n_{k}},v_{n_{k}}\right\} \right\}_{k\geqslant 2}$ is a sequence satisfying  Lemma \ref{Corollary1}\ref{Corollary1a}. Then
		\begin{align}\label{remarkable lemma11}
		\begin{split}
		\frac{f(u_{n_{k}})-f(a)}{u_{n_{k}}-a}&=\frac{f_{n_{k}}(u_{n_{k}})}{u_{n_{k}}-a}=
		\frac{g_{n_{k}}(u_{n_{k}})}{u_{n_{k}}-a}\\
		&=\frac{\left(u_{n_{k}}-a_{n_{k}}\right)^{\frac{1}{4}}\sin\left(u_{n_{k}}-a_{n_{k}}\right)^{-\frac{7}{4}}}{u_{n_{k}}-a}=\frac{\left(u_{n_{k}}-a_{n_{k}}\right)^{\frac{1}{4}}}{u_{n_{k}}-a}.
		\end{split}
		\end{align}
		Now as
		\[
		\frac{b_{n_{k}}-a_{n_{k}}}{4}=d_{n_{k}}-a_{n_{k}}<u_{n_{k}}-a_{n_{k}}<c_{n_{k}}-a_{n_{k}}=\frac{b_{n_{k}}-a_{n_{k}}}{2},
		\]
		using \eqref{nondifferentiableinequality}, we get
		\begin{align}\label{remarkable lemma12}
		\frac{1}{2^{n_{k}+\frac{1}{2}}}<\left(u_{n_{k}}-a_{n_{k}}\right)^{\frac{1}{4}}<\frac{1}{2^{\frac{n_{k}}{2}+\frac{3}{4}}}.
		\end{align}
		Furthermore, as $\left(a_{n_{k}},b_{n_{k}}\right)\subseteq \left[a+1/4^{n_{k}+1}, a+5/4^{n_{k}+1}\right]$,
		\begin{align}\label{remarkable lemma13}
		\frac{1}{4^{n_{k}+1}}<u_{n_{k}}-a<\frac{5}{4^{n_{k}+1}}.
		\end{align}
		Finally, applying \eqref{remarkable lemma12} and \eqref{remarkable lemma13} in \eqref{remarkable lemma11}, we get $$\lim\limits_{k\rightarrow \infty}\frac{f(u_{n_{k}})-f(a)}{u_{n_{k}}-a}=\infty.$$ 
		Proof of $\lim\limits_{k\rightarrow \infty}(f(v_{n_{k}})-f(a))/(v_{n_{k}}-a)=0$
		is obvious. Therefore $f$ is not differentiable at $a$ and this completes the proof.
	\end{proof}
	
	\begin{corollary}\label{Not ACG^{*}}
		Let $f$ be the function defined in \eqref{Eq13}. Then $f$ is not $ACG^{*}$.
	\end{corollary}
	\begin{proof}
		Since $ACG^{*}$ functions are differentiable a.e., by Theorem \ref{remarkable lemma1}, we get that $f$ is not $ACG^{*}$.
	\end{proof}

	\begin{theorem}\label{remarkable lemma2}
		Let $f$ be the function defined in \eqref{Eq13}. Then $f$ is Laplace differentiable on $[0,1]$.
	\end{theorem}
	\begin{proof}
		It is sufficient to check the Laplace differentiability of $f$ on $S$ only. Let $x\in S\setminus\{1\}$. Choose $\delta>0$ be such that $x+\delta\in S$. Then
		\begin{align*}
		[x,x+\delta]=\left([x,x+\delta]\setminus S\right)\cup \left([x,x+\delta]\cap S\right)=\left(\cup_{k}\left(a_{n_{k}},b_{n_{k}}\right)\right)\cup A_{S},
		\end{align*}
		where $A_{S}=[x,x+\delta]\cap S$ and $\left\{\left(a_{n_{k}},b_{n_{k}}\right)\right\}$ is a sub-sequence of $\left\{\left(a_{n},b_{n}\right)\right\}$, the sequence of contiguous intervals of $S$ in $[0,1]$. Then
		\begin{gather*}
		\int_{x}^{x+\delta}s^{2}e^{-s(y-x)}\left(f(y)-f(x)\right)\,dy=\int_{x}^{x+\delta}s^{2}e^{-s(y-x)}f(y)\,dy\\
		=\int\limits_{\bigcup_{k}\left(a_{n_{k}},b_{n_{k}}\right)} s^{2}e^{-s(y-x)}f(y)\,dy=\sum_{k}\int_{a_{n_{k}}}^{b_{n_{k}}}s^{2}e^{-s(y-x)}f_{n_{k}}(y)\,dy.
		\end{gather*}
		Choose an interval $\left(a_{n_{k}},b_{n_{k}}\right)$ from $\left\{\left(a_{n_{k}},b_{n_{k}}\right)\right\}$ and to avoid the complexity of notations denote it by $\left(a_{n},b_{n}\right)$. Note that
		\begin{align}
		&a_{n}\leqslant y\leqslant c_{n}\implies
		\begin{cases}
		0\leqslant (y-a_{n})< (y-x)\,\,\text{and}\\
		0<(c_{n}-a_{n})<(c_{n}-x)<\delta.
		\end{cases}\label{eq:facts1}\\
		&c_{n}\leqslant y\leqslant b_{n}\implies
		\begin{cases}
		0\leqslant (b_{n}-y)< (y-x)\,\,\text{and}\\
		0<(b_{n}-c_{n})<(c_{n}-x)<\delta.
		\end{cases}\label{eq:facts3}\\
		&\quad\qquad\qquad\qquad\left|G_{n}(x)\right|\leqslant (x-a_{n})^{3}.\label{eq:facts2}\\
		&\quad\qquad\qquad\qquad\left|H_{n}(x)\right|\leqslant (b_{n}-x)^{3}.\label{eq:facts4}
		\end{align}
		
		As $x^{r}e^{-x}\leqslant r!$ for $x>0$, by \eqref{eq:facts1} and \eqref{eq:facts2}, we get
		\begin{align}
		\begin{split}
		&\left|\int_{a_{n}}^{c_{n}}s^{2}e^{-s(y-x)}G_{n}^{'}(y)\,dy\right|\\
		&\leqslant \left|s^{2}e^{-s(c_{n}-x)}G_{n}(c_{n})\right| + \left|\int_{a_{n}}^{c_{n}}s^{3}e^{-s(y-x)}G_{n}(y)\,dy\right|\\
		&\leqslant s^{2}e^{-s(c_{n}-x)}(c_{n}-a_{n})^{3}+\int_{a_{n}}^{c_{n}}s^{3}e^{-s(y-x)}(y-a_{n})^{3}\,dy\\
		&\leqslant s^{2}e^{-s(c_{n}-x)}(c_{n}-x)^{2}(c_{n}-a_{n})+\int_{a_{n}}^{c_{n}}s^{3}e^{-s(y-x)}(y-x)^{3}\,dy\\
		&\leqslant 2(c_{n}-a_{n})+ 6(c_{n}-a_{n})=8(c_{n}-a_{n}).\label{eq:Facts1}
		\end{split}
		\end{align}
		Furthermore, \eqref{Eq7}, \eqref{eq:facts1} and \eqref{eq:Facts1} imply
		\begin{align*}
		&\left|\int_{a_{n}}^{c_{n}}s^{2}e^{-s(y-x)}f_{n}(y)\,dy\right|=\left|\int_{a_{n}}^{c_{n}}s^{2}e^{-s(y-x)}g_{n}(y)\,dy\right|\\
		&\leqslant\left|\int_{a_{n}}^{c_{n}}s^{2}e^{-s(y-x)}\left(\frac{4}{7}G_{n}^{'}(y)-\frac{12}{7}(y-a_{n})^{2}\cos(y-a_{n})^{-\frac{7}{4}}\right)\,dy\right|\\
		&\leqslant\frac{32}{7}(c_{n}-a_{n}) + \frac{12}{7}\int_{a_{n}}^{c_{n}}s^{2}e^{-s(y-x)}(y-a_{n})^{2}\,dy\\
		&\leqslant\frac{32}{7}(c_{n}-a_{n}) + \frac{12}{7}\int_{a_{n}}^{c_{n}}s^{2}e^{-s(y-x)}(y-x)^{2}\,dy\leqslant 8(c_{n}-a_{n}).
		\end{align*}
		Similarly, $\left|\int_{c_{n}}^{b_{n}}s^{2}e^{-s(y-x)}f_{n}(y)\,dy\right|\leqslant 8(b_{n}-c_{n})$ can be proved with the help of \eqref{Eqnew10}, \eqref{eq:facts3} and \eqref{eq:facts4}. It follows
		\[
		\left|\int_{a_{n}}^{b_{n}}s^{2}e^{-s(y-x)}f_{n}(y)\,dy\right|\leqslant 8(b_{n}-a_{n}).
		\]
		Thus
		\begin{align}\label{equation 4.25}
		\begin{split}
		\limsup\limits_{s\rightarrow \infty}\left|\int_{x}^{x+\delta}s^{2}e^{-s(y-x)}f(y)\,dy\right|&\leqslant \limsup\limits_{s\rightarrow \infty} \sum_{k}\left|\int_{a_{n_{k}}}^{b_{n_{k}}}s^{2}e^{-s(y-x)}f_{n_{k}}(y)\,dy\right|\\
		&\leqslant \limsup\limits_{s\rightarrow \infty} \sum_{k}8(b_{n_{k}}-a_{n_{k}})\leqslant \delta.
		\end{split}
		\end{align}
		As $x$ is arbitrary, by\eqref{equation 4.25} and Lemma \ref{lemma11}, we get $LD_{1}^{+}f(x)=0$ for all $x\in S\setminus\{1\}$. Proof of $LD_{1}^{-}f(x)=0$ for all $x\in S\setminus\{0\}$ is similar. This completes the proof.
	\end{proof}

	\begin{theorem}\label{remarkable theorem}
		Let $f$ be the function defined in \eqref{Eq13}. Then $LD_{1}f$ is Laplace integrable on $[0,1]$ but not Perron integrable.
	\end{theorem}
	\begin{proof}
		By Theorem \ref{remarkable lemma2}, we get $LD_{1}f$ exists finitely on $[0,1]$. Let $\phi= LD_{1}f$. As $f(0)=0$, Theorem \ref{THEOREM} implies $\phi$ is Laplace integrable and $f(x)=\int_{0}^{x}\phi$ for $x\in[0,1]$. Now if $\phi$ is Perron integrable, then $(\mathsf{P})\int_{0}^{x}\phi=\int_{0}^{x}\phi$ for all $x\in [0,1]$. Which implies $f$ is $ACG^{*}$, contradicting the Corollary \ref{Not ACG^{*}}. Thus $\phi$ is not Perron integrable.
	\end{proof}

	\section{Basic properties}\label{Basic Properties}

	\par The next theorem is a straightforward consequence of the definition of Laplace integral, so we omit its proof.

	\begin{theorem}\label{Theorem 1}\noindent
		\begin{enumerate}[\upshape A.]
			\item Let $f\in \mathcal{LP}[a,b]$ and $F(x)= \int_{a}^{x}f$. If $U$ and $V$ are major and minor functions of $f$ respectively, then each of the functions $U-F$ and $F-V$ are non-decreasing. \label{theorem 3.8}
			\item 
			\begin{enumerate}[\upshape a.]
				\item If $f\in \mathcal{LP}[a,b]$, the $f$ is Laplace integrable on every sub-interval of $[a,b]$.
				\item Let $c\in (a,b)$. If $f$ is Laplace integrable on both intervals $[a,c]$ and $[c,b]$, then $f$ is Laplace integrable on $[a,b]$ and $\int_{a}^{b}f =\int_{a}^{c}f + \int_{c}^{b}f$.
			\end{enumerate}
			\item If $f\in \mathcal{LP}[a,b]$, then $f$ is finite valued a.e on $[a,b]$.
			\item \label{A.E. Equal Functions} Let $f\in \mathcal{LP}[a,b]$ and let $g=f$ a.e. on $[a,b]$. Then $g\in \mathcal{LP}[a,b]$ and $\int_{a}^{b}g=\int_{a}^{b}f$.
			\item Let $f, h \in \mathcal{LP}[a,b]$. Then:
			\begin{enumerate}[\upshape a.]
				\item $kf+ lh\in \mathcal{LP}[a,b]$ and $\int_{a}^{b}(kf+ lh)= k\int_{a}^{b}f + l\int_{a}^{b}h$ for each $k,\,l\in \mathbb{R}$.
				\item If $f\leqslant h$ a.e. on $[a,b],$ then $\int_{a}^{b}f\leqslant \int_{a}^{b}h$.
			\end{enumerate}
		\end{enumerate}
	\end{theorem}

	\begin{theorem}\label{Continuous primitive}
		Let $f\in \mathcal{LP}[a,b]$ and let $F(x)= \int_{a}^{x}f$ for all $x\in [a,b]$. Then $F$ is continuous on $[a,b]$.
	\end{theorem}
	\begin{proof}
		Let $[x_{1}, x_{2}]\subseteq [a,b]$ and let $U$, $V$ be major and minor functions of $f$ respectively on $[a,b]$. Then by Theorem \ref{Theorem 1}\ref{theorem 3.8}, we get $V_{x_{1}}^{x_{2}}\leq F_{x_{1}}^{x_{2}}\leq U_{x_{1}}^{x_{2}}$, and this implies $f$ is uniformly continuous on $[a,b]$.
	\end{proof}

	\begin{theorem}[\bf Fundamental theorem of calculus I]\label{THEOREM}
		If $F\colon[a,b]\to \mathbb{R}$ is continuous and $LD_{1}F$ exists finitely nearly everywhere on $[a,b]$, then $LD_{1}F\in \mathcal{LP}[a,b]$ and $\int_{a}^{x}LD_{1}F = F(x) - F(a)$ for all $x\in [a,b]$.
	\end{theorem}
	\begin{proof}
		Let $LD_{1}F = f$ nearly everywhere. Then $F$ is a major as well as a minor function of $f$. Which implies $F(b) - F(a)\leqslant \underline{\int_{a}^{b}}f \leqslant \overline{\int_{a}^{b}}f \leqslant F(b) - F(a)$. Therefore $\int_{a}^{b}LD_{1}F$ exists, and $\int_{a}^{b}LD_{1}F = F(b) - F(a)$.
	\end{proof}
	
	\begin{theorem}\label{THEOREm}
		Let $LD_{1}F=f$ on $[a,b]$. If $f\in \mathcal{LP}[a,b]$, then $\int_{a}^{x}f = F(x) - F(a)$ for all $x\in [a,b]$.
	\end{theorem}
	\begin{proof}
		It is evident that $F$ is Laplace continuous on $[a,b]$. Now using Theorem 3.7 of \cite{OLC}, it can be easily proved that $F$ is a primitive of $f$.
	\end{proof}

	\begin{theorem}[\bf Fundamental theorem of calculus II]\label{A.E. Laplace Differentiability}
		Let $f\in \mathcal{LP}[a,b]$ and let $F(x)=\int_{a}^{x}f$, $x\in [a,b]$. Then $LD_{1}F=f$ a.e. on $[a,b]$.
	\end{theorem}
	\begin{proof}
		Suppose $LD_{1}F\neq f$ on a positive measure subset of $[a,b]$. Then $\exists$ a positive constant $k$ such that at least one of the following two sets
		\begin{align*}
		E_{1}&=\left\{ x\in [a,b]\mid f(x)-\underline{LD}_{1}F(x)>k     \right \}\\
		E_{2}&=\left\{ x\in [a,b]\mid \overline{LD}_{1}F(x)-f(x)>k     \right\}
		\end{align*}
		is of positive measure. Assume $\lambda(E_{2})=\mu>0$. Let $0<\epsilon<\frac{1}{2}k\mu$ and let $V$ be a minor function of $f$ such that $(F-V)_{a}^{x}<\epsilon$ for all $x\in [a,b]$. Define $R=F-V$ on $[a,b]$. Then $R$ is non-decreasing and $\int_{a}^{b}R'\leqslant R(b)-R(a)<\epsilon$. The set $E_{3}=\left\{ x\in [a,b]\mid R'(x)>\frac{1}{2}k  \right\}$ is of measure less than $\mu$; otherwise, we will get $\epsilon>\frac{1}{2}k\mu$, which is a contradiction. Hence $E_{2}\setminus E_{3}$ is nonempty. Now for $x\in E_{2}\setminus E_{3}$, we have $f(x)\geqslant \overline{LD}_{1}V(x)=\overline{LD}_{1}F(x)-R'(x)\geqslant \overline{LD}_{1}F(x)-\frac{1}{2}k$, which is again a contradiction. Therefore $LD_{1}F=f$ a.e. on $[a,b]$.
	\end{proof}

	\begin{corollary}
		If $f\in \mathcal{LP}[a,b]$, then $f$ is measurable.
	\end{corollary}
	\begin{proof}
		Define $F_{n}(x)= n^{2}\int_{0}^{\delta}e^{-nt}[F(x+t)-F(x)]$. Then $F_{n}$ is continuous on $[a,b]$ and $\lim\limits_{n\rightarrow \infty}F_{n}(x) = f(x)$ a.e. on $[a,b]$, which implies $f$ is measurable.
	\end{proof}

	\par In \cite{OLC}, the Laplace continuity is defined for Perron integrable functions. Here we redefine the Laplace continuity $f$ at $x$ by assuming it is Laplace integrable.

	\begin{definition}\label{Laplace continuity}
		Let $f$ be Laplace integrable on a neighbourhood of $x$. If $\exists\, \delta>0$ such that the following limits
		\begin{align*}
		\lim\limits_{s\rightarrow \infty}s\int_{0}^{\delta}e^{-st}f(x+t)\,dt\quad\text{and}\quad\lim\limits_{s\rightarrow \infty}s\int_{0}^{\delta}e^{-st}f(x-t)\,dt
		\end{align*}
		exist and are equal, then the common value is denoted by $LD_{0}f(x)$. And we say $f$ is Laplace continuous at $x$ if $LD_{0}f(x)=f(x)$.
	\end{definition}

	\par The above definition also does not depend on $\delta$. It can be easily proved that the continuity of $f$ at $x$ implies $LD_{0}f(x)=f(x)$.

	\begin{theorem}\label{Continuity implies differentiability}
		Let $f\in \mathcal{LP}[a,b]$ and let $F(x)=\int_{a}^{x}f$.
		\begin{enumerate}[\upshape A.]
			\item If $f$ is Laplace continuous at $x\in [a,b]$, then $LD_{1}F(x)=f(x)$.
			\item If $f$ is continuous at $x\in [a,b]$, then $F^{'}(x)=f(x)$
		\end{enumerate}
	\end{theorem}

	\par Proof of the first part is similar to that of Theorem 3.2 of \cite{OLC} and the second part is straightforward. Let $I$ be an interval and let $\mathcal{BLC}(I)$ be the set of all bounded Laplace continuous functions on $I$. As bounded Laplace integrable functions are Lebesgue integrable, by the Corollary 3.13 of \cite{OLC}, we get $\left( \mathcal{BLC}(I), \|\,.\,\|_{\infty} \right)$ is complete, where $\|f\|_{\infty}=\sup\limits_{x\in I}|f(x)|$. Now define
	\[
	\mathcal{BLC}^{n}(I)=\left\{ (f_{1},...,f_{n})^{T}\mid f_{i}\in\mathcal{BLC}(I)\quad\text{for each $i=1,...,n$}  \right\}
	\]
	and
	\[
	\|(f_{1},...,f_{n})^{T}\|_{n,\infty}=\max\limits_{1\leqslant i \leqslant n}\left\{\|f_{i}\|_{\infty}\right\},
	\]
	where by $(f_{1},...,f_{n})^{T}$ we mean the transpose of $(f_{1},...,f_{n})$. Then we get
	\begin{theorem}\label{Complete-bounded-Laplace}
		Let $I$ be an interval. Then $\left( \mathcal{BLC}^{n}(I), \|\,.\,\|_{n,\infty} \right)$ is complete. 
	\end{theorem}
	\par Let $F(x)=\int_{a}^{x}f$ and let $\mathcal{LP}^{*}[a,b]$ be the set of all equivalence classes of Laplace integrable functions, where the equivalence relation is ``$f\sim g$ iff $f=g$ a.e. on $[a,b]$". Define $\|f\|_{[a,b]}=\|F\|_{\infty}$ for $f\in \mathcal{LP}[a,b]$. Then $(\mathcal{LP}^{*}[a,b],\|\,.\,\|_{[a,b]})$ forms a normed linear space. The norm, $\|\,.\,\|_{[a,b]}$, is generally known as the Alexiewicz's norm (\cite{LFDIF}). For convenience, we will denote $\mathcal{LP}^{*}[a,b]$ by $\mathcal{LP}[a,b]$ and assume that it will not confuse.

	\begin{theorem}\label{incomplete}
		$(\mathcal{LP}[0,1],\|\,.\,\|_{[0,1]})$ is not complete.
	\end{theorem}
	\begin{proof}
		Let $P$ be the Cantor middle third subset in $[0,1]$ and let $\theta$ be the Cantor singular function on $[0,1]$. Let $I_{nk}, k=1,...,2^{n-1}$, be the open intervals, which are removed from $[0,1]$ in the nth step to construct $P$. Let
		\begin{gather*}
		G_{n}=\bigcup\limits_{j=1}^{n}\bigcup\limits_{k=1}^{2^{j-1}}I_{jk}\quad\text{and}\quad H_{n}=[0,1]\setminus G_{n}=\bigcup\limits_{i=1}^{2^{n}}[a_{i},b_{i}],
		\end{gather*}
		where the intervals $\displaystyle [a_{i},b_{i}]$ are the disjoint closed sub-intervals of the closed set $H_{n}$. Now define for each $n\in\mathbb{N}$
		\begin{align*}
		f_{n}(x)=
		\begin{cases}
		\qquad\qquad\qquad\theta(x)&\text{if $x\in G_{n}$}\\
		\theta(a_{i})+[\theta(b_{i})-\theta(a_{i})]\omega\left(\dfrac{x-a_{i}}{b_{i}-a_{i}}\right)&\text{if $x\in[a_{i},b_{i}],\quad i=1,...,2^{n}$},
		\end{cases}
		\end{align*}
		where
		\[
		\omega(x)=\dfrac{2x^{2}(3-2x)+\int_{0}^{x}t(1-t)\sin(1/t)\,dt}{2+\int_{0}^{1}t(1-t)\sin(1/t)\,dt}.
		\]
		Then from Theorem 3.4 of \cite{INTRSE}, we get $f_{n}^{'}$ exists on $[0,1]$, $f_{n}(x)=(\mathsf{P})\int_{0}^{x}f_{n}^{'}$ and $\|f_{n}-\theta\|_{\infty}\rightarrow 0$ as $n\rightarrow\infty$. Thus $\{ f_{n}^{'} \}$ is a Cauchy sequence in $(\mathcal{LP}[0,1],\|\,.\,\|_{[0,1]})$. Now if $\{ f_{n}^{'} \}$ converges to some $g\in \mathcal{LP}[0,1]$, then $\int_{0}^{x}g=\theta(x)$ for all $x\in[0,1]$. As $\theta^{'}=0$ a.e. on $[0,1]$, by Theorem \ref{A.E. Laplace Differentiability}, we get $LD_{1}\theta=\theta^{'}=g=0$ a.e. on $[0,1]$. Which implies $\theta=0$ on $[0,1]$ (Theorem \ref{Theorem 1}\ref{A.E. Equal Functions}), a contradiction. Hence $\{ f_{n}^{'} \}$ cannot converge in $(\mathcal{LP}[0,1],\|\,.\,\|_{[0,1]})$ and this completes the proof.
	\end{proof}

	\begin{theorem}\label{Denseness of plynomials}
		Let $\mathcal{P}[a,b]$ be the set of all polynomials on $[a,b]$. Then $\mathcal{P}[a,b]$ is dense in $\left(\mathcal{LP}[a,b],\|\,.\,\|_{[a,b]}\right)$.
	\end{theorem}
	\begin{proof}
		Let $f\in \mathcal{LP}[a,b]$ and let $F(x)=\int_{a}^{x}f$. As $F$ is continuous on $[a,b]$, by Weierstrass approximation theorem, we can choose a sequence of polynomials converging to $F$ uniformly on $[a,b]$. Let $\left\{q_{n}\right\}$ be such a sequence in $\mathcal{P}[a,b]$. It is evident that $\{q_{n}^{'}\}$ is also a sequence in $\mathcal{P}[a,b]$. Moreover, as $\|F-q_{n}\|_{\infty}=\|f-q_{n}^{'}\|_{[a,b]}$, $\{q_{n}^{'}\}$ converges to $f$ in $\left(\mathcal{LP}[a,b],\|\,.\,\|_{[a,b]}\right)$. This completes the proof.
	\end{proof}

	\begin{theorem}\label{Non-negative}
		If $f\in \mathcal{LP}[a,b]$ and  $f\geqslant 0$ a.e. on $[a,b]$, then $f$ is Lebesgue integrable.
	\end{theorem}
	\begin{proof}
		Let $U$ be a major function for $f$. Then $U$ is continuous, $\underline{LD}_{1}U\geqslant f\geqslant 0$ a.e. on $[a,b]$ and $\underline{LD}_{1}U> -\infty$ on $[a,b]$. So by Theorem $19$ of \cite{OLD}, $U$ is non-decreasing and $U'$ exists finitely a.e. on $[a,b]$. Thus $U'$ is Lebesgue integrable on $[a,b]$ and $U'(x)\geqslant f(x)\geqslant 0$ a.e. on $[a,b]$, which implies $f$ is Lebesgue integrable  on $[a,b]$.
	\end{proof}

	\begin{corollary}
		If $f\colon[a,b]\to \mathbb{R}$ is bounded and Laplace integrable, then $f$ is Lebesgue integrable.
	\end{corollary}

	\begin{corollary}
		If $f\colon[a,b]\to \mathbb{R}$ is Laplace integrable on every measurable subset of $[a,b]$, then $f$ is Lebesgue integrable on $[a,b].$
	\end{corollary}

	\begin{theorem}[\bf Hake's theorem for Laplace integral]
		Let $f\colon[a,b]\to\mathbb{R}$. Then $f\in \mathcal{LP}[a,b]$ iff for every $c\in (a,b)$, $\left.f\right|_{[a,c]}\in \mathcal{LP}[a,c]$ and there exists $L\in \mathbb{R}$ such that $\lim\limits_{c\rightarrow b-}\int_{a}^{c}f=L$. In this case, we write $\int_{a}^{b}f=L$.
	\end{theorem}
	
	\par The ``only if part'' of this theorem is similar to the ``only if part'' of Theorem 12.8 of \cite[p.~195]{AMTOIntegration}. For the ``if part'' see Theorem 8.18 of \cite{RSG}.

	\section{Integration by parts}

	\begin{lemma}\label{lemmaforintp2}
		Let $f$ be continuous at $x_{0}$ and let $g$ be differentiable at $x_{0}$. If $g({x_{0}})\geqslant 0$, then
		\begin{enumerate}[\upshape A.]
			\item $\underline{LD}_{1}(fg)(x_{0})\geqslant g(x_{0})\underline{LD}_{1}f(x_{0})+f(x_{0})g'(x_{0})$,
			\item $\overline{LD}_{1}(fg)(x_{0})\leqslant g(x_{0})\overline{LD}_{1}f(x_{0})+f(x_{0})g'(x_{0})$,
		\end{enumerate}
	\end{lemma}
	\begin{proof}
		We only give proof of the first part, and the second part is similar. Furthermore, to prove the first part it is sufficient to show that
		\[
		\underline{LD}_{1}^{+}(fg)(x_{0})\geqslant g(x_{0})\underline{LD}_{1}^{+}f(x_{0})+f(x_{0})g'(x_{0}).
		\] 
		\par As $f$ is continuous and $g$ is differentiable at $x_{0}$, for $\epsilon>0$, $\exists\delta>0$ such that $0<t<\delta$ implies, $\left| f(x_{0}+t)-f(x_{0}) \right|<\epsilon$ and $\left| g(x_{0}+t)-g(x_{0})-tg'(x_{0}) \right|<\epsilon t$. Thus
		\begin{align*}
		\limsup\limits_{s\rightarrow \infty}\, s^{2}\int_{0}^{\delta}e^{-st}&\left| f(x_{0}+t)\left[ g(x_{0}+t)-g(x_{0})-tg'(x_{0}) \right] \right|\,dt\\
		&\leqslant\epsilon (\left|f(x_{0})\right|+\epsilon)\limsup\limits_{s\rightarrow \infty}\, s^{2}\int_{0}^{\delta}e^{-st}t\,dt=\epsilon (\left|f(x_{0})\right|+\epsilon)
		\end{align*}
		As $\epsilon$ is arbitrary, Lemma \ref{lemma11} implies
		\begin{align}\label{equation for INBP1}
		\lim\limits_{s\rightarrow \infty} s^{2}\int_{0}^{\delta}e^{-st}f(x_{0}+t)\left[ g(x_{0}+t)-g(x_{0})-tg'(x_{0}) \right]\,dt=0.
		\end{align}
		Similarly,
		\begin{align}\label{equation for INBP2}
		\lim\limits_{s\rightarrow \infty} s^{2}\int_{0}^{\delta}e^{-st}tg'(x_{0})\left[ f(x_{0}+t)-f(x_{0}) \right]\,dt=0.
		\end{align}
		Note that
		\begin{align}\label{equation for INBP3}
		\begin{split}
		(fg)(x_{0}+t)&-(fg)(x_{0})=g(x_{0})\left[ f(x_{0}+t)-f(x_{0})\right]\\
		&+ f(x_{0}+t)\left[ g(x_{0}+t)-g(x_{0})-tg'(x_{0}) \right]\\
		&+ tg'(x_{0})\left[ f(x_{0}+t)-f(x_{0}) \right] + tg'(x_{0})f(x_{0}).
		\end{split}
		\end{align}
		Now applying \eqref{equation for INBP1} and \eqref{equation for INBP2} in \eqref{equation for INBP3}, we get \[\underline{LD}_{1}^{+}(fg)(x_{0})\geqslant g(x_{0})\underline{LD}_{1}^{+}f(x_{0})+f(x_{0})g'(x_{0}).\qedhere\]
	\end{proof}
	\begin{corollary}\label{Laplace differentiability of product}
		Let $f$ be continuous and Laplace differentiable at $x_{0}$ and $g$ be differentiable at $x_{0}$. Then $fg$ is Laplace differentiable at $x_{0}$. Moreover, $LD_{1}fg=g(x_{0})LD_{1}f(x_{0})+f(x_{0})g^{'}(x_{0})$.
	\end{corollary}

	\begin{theorem}[\bf Integration by parts]\label{IntByParts1}
		Let $f\in \mathcal{LP}[a,b]$ and let $g\in\mathcal{BV}[a,b]$, where $\mathcal{BV}[a,b]$ is the set of all functions of bounded variation defined over $[a,b]$. Then $fG\in \mathcal{LP}[a,b]$ and $\int_{a}^{b}fG=F(b)G(b)-(\mathsf{P})\int_{a}^{b}Fg$, where $F(x)=\int_{a}^{x}f$, $G(x)=\int_{a}^{x}g$ and $(\mathsf{P})\int_{a}^{b}$ is the definite Perron integral.
	\end{theorem}
	\begin{proof}
		Let $g=g_{1}-g_{2}$, where $g_{1}$ and $g_{2}$ are two non-negative and non-decreasing functions. Then $G=G_{1}-G_{2}$, where $G_{i}(x)=\int_{a}^{x}g_{i}$ for $i=1,2$. Now it is sufficient to prove this theorem for the pair $g_{1}$, $G_{1}$.
		\par Note that, $G_{1}$ is non-negative on $[a,b]$ and $G'_{1}=g_{1}$ nearly everywhere on $[a,b]$. Let $U$ and $V$ be a major and a minor function of $f$ respectively, with $V(a)=U(a)=0$. Then by Lemma \ref{lemmaforintp2}, $UG_{1}$ and $VG_{1}$ are a major and a minor function of $Fg+fG$  respectively. Since $f$ is Laplace integrable, by Corollary \ref{corollary_1}, for any $\epsilon>0$ we can choose $U$ and $V$ in such a way that $U(b)-V(b)<\epsilon$. For these $U$ and $V$, we get $(UG_{1})(b)-(VG_{1})(b)=G_{1}(b)(U(b)-V(b))<\epsilon G_{1}(b)$ and $(VG_{1})(b)\leqslant F(b)G_{1}(b)\leqslant (UG_{1})(b)$. Thus Corollary \ref{corollary_1} implies $fG+ Fg\in \mathcal{LP}[a,b]$ and $\int_{a}^{b}(fG+Fg)=F(b)G(b)$. Now as $Fg\in\mathsf{P}[a,b]$, we have $\int_{a}^{b}fG=F(b)G(b)- (\mathsf{P})\int_{a}^{b}Fg$.
	\end{proof}

	\begin{corollary}
		Let $f\in \mathcal{LP}[a,b]$ and let $F(x)=\int_{a}^{x}f$ for $x\in [a,b]$. Then $f$ is Laplace continuous a.e. on $[a,b]$.
	\end{corollary}
	
	\begin{proof}
		Note that for $\delta>0$,
		\begin{align*}
		s\int_{0}^{\delta}e^{-st}&f(x+t)\,dt\\
		&= \left[se^{-st}F(x+t)\right]^{\delta}_{0}+s^{2}\int_{0}^{\delta}e^{-st}[F(x+t)-F(x)]\,dt + s^{2}\int_{0}^{\delta}e^{-st}F(x)\,dt\\
		&=se^{-s\delta}[F(x+\delta)-F(x)] + s^{2}\int_{0}^{\delta}e^{-st}[F(x+t)-F(x)]\,dt.
		\end{align*}
		Now by Theorem \ref{A.E. Laplace Differentiability}, we obtain that $LD^{+}_{0}f(x)=f(x)$ for a.e. $x\in [a,b]$. Similarly, we can show $LD^{-}_{0}f(x)=f(x)$ for a.e. $x\in [a,b]$ and this completes the proof.
	\end{proof}
	
	\begin{theorem}
		Let $f\in \mathcal{LP}[a,b]$ and let $g\in \mathcal{BV}[a,b]$, then
		\[
		\left| \int_{a}^{b}fG \right|\leqslant \left| \int_{a}^{b}f \right|\inf\limits_{x\in[a,b]}\left| G(x)\right| + \|f\|_{[a,b]}V_{G}[a,b].
		\]
	\end{theorem}
	
	\par The proof of this theorem is similar to that of Lemma 24 of \cite{HKFT}.

	\section{Convergence theorems}
	
	\par The main theorem of \cite{ONTCSPI} gives a necessary and sufficient condition for the convergence of Perron integral. A similar type of theorem also exists for Laplace integral, which is given below. Furthermore, the dominated convergence theorem appears as a corollary of this theorem.

	\begin{theorem}\label{convergencetheorem1}
		Let $\left\{f_{n}\right\}$ be a sequence of Laplace integrable functions defined over $[a,b]$ such that
		\begin{enumerate}[\upshape A.]
			\item for all $n\in \mathbb{N}$, $f_{n}\geqslant g$ a.e. on $[a,b]$ where $g$ is Laplace integrable on $[a,b]$ and
			\item $\lim\limits_{n\rightarrow\infty}f_{n}=f$ a.e. on $[a,b]$.
		\end{enumerate}
		Then $f$ is Laplace integrable on $[a,b]$ and $\lim\limits_{n\rightarrow\infty}\int_{a}^{x}f_{n}=\int_{a}^{x}f$ if and only if the sequence of integrals $\left\{ \int_{a}^{x}(f_{n}-g)  \right\}$ is equiabsolutely continuous \textup{(EAC)} on $[a,b]$.
	\end{theorem}
	\begin{proof}
		Note that $f_{n}-g$ is Perron integrable. Thus by the Theorem of \cite{ONTCSPI}, we get $f-g\in \mathsf{P}[a,b]$ and $\lim\limits_{n\rightarrow\infty}(\mathsf{P})\int_{a}^{x}(f_{n}-g)=(\mathsf{P})\int_{a}^{x}(f-g)$ iff the sequence of integrals $\left\{ (\mathsf{P})\int_{a}^{x}(f_{n}-g)  \right\}$ is EAC on $[a,b]$. But as $g\in \mathcal{LP}[a,b]$, we get $f\in \mathcal{LP}[a,b]$ and $\lim\limits_{n\rightarrow\infty}\int_{a}^{x}f_{n}=\int_{a}^{x}f$ iff $\left\{ \int_{a}^{x}(f_{n}-g)  \right\}$ is EAC on $[a,b]$
	\end{proof}

	\begin{corollary}[\bf Dominated convergence theorem]\label{Dominated convergence theorem}
		Let $\left\{f_{n}\right\}$ be a sequence of Laplace integrable functions defined over $[a,b]$ such that
		\begin{enumerate}[\upshape A.]
			\item for all $n\in \mathbb{N}$, $g\leqslant f_{n}\leqslant h$ a.e. on $[a,b]$, where $g$ and $h$ are Laplace integrable on $[a,b]$ and
			\item $\lim\limits_{n\rightarrow\infty}f_{n}=f$ a.e. on $[a,b]$.
		\end{enumerate}
		Then $f$ is Laplace integrable on $[a,b]$ and $\lim\limits_{n\rightarrow\infty}\int_{a}^{x}f_{n}=\int_{a}^{x}f$, $\forall x \in [a,b]$.
	\end{corollary}
	\begin{proof}
		Let $F_{n}(x)=\int_{a}^{x}(f_{n}-g)$, $x\in [a,b]$. As $0\leqslant f_{n}-g\leqslant h-g$ and $\phi(x)=\int_{a}^{x}(h-g)$ is absolutely continuous on $[a,b]$, $\left\{ F_n \right\}$ is EAC on $[a,b]$. This completes the proof.
	\end{proof}

	\begin{theorem}[\bf Sequence of integrals]\label{sequence of integrals}
		Let $f\in \mathcal{LP}[a,b]$ and let $\left\{g_{n}\right\}$ be a sequence in $\mathcal{BV}[a,b]$ such that $\{g_{n}\}$ is of uniform bounded variation and converges pointwise to some $g$ on $[a,b]$, then $\lim\limits_{n\rightarrow \infty}\int_{a}^{b}fG_{n}=\int_{a}^{b}fG$, where $G_{n}(x)=\int_{a}^{x}g_{n}$ and $G(x)=\int_{a}^{x}g$.
	\end{theorem}
	\begin{proof}
		As $\{g_{n}\}$ is of uniform bounded variation, $g\in \mathcal{BV}[a,b]$. Thus $\int_{a}^{b}fG$ is well defined. Now by Corollary 3.2 of \cite{LimitsHenstock}, we get $\lim\limits_{n\rightarrow \infty}\int_{a}^{b}g_{n}=\int_{a}^{b}g$ and $\lim\limits_{n\rightarrow \infty}\int_{a}^{b}Fg_{n}=\int_{a}^{b}Fg$, where $F(x)=\int_{a}^{x}f$. Thus
		\begin{align*}
		\lim\limits_{n\rightarrow \infty}\int_{a}^{b}fG_{n}&=\lim\limits_{n\rightarrow \infty}\left\{ F(b)G_{n}(b)-\int_{a}^{b}Fg_{n} \right\}\\
		&=F(b)G(b)-\int_{a}^{b}Fg=\int_{a}^{b}fG.\qedhere
		\end{align*}
	\end{proof}

	\section{Mean value theorems}\label{Mean value theorems}

	\par  The first mean value theorem for Laplace integral given below is an extension of the generalised mean value theorem stated in \cite{TTUGRIntegral}.

	\begin{theorem}[\bf The first mean value theorem]\label{The first mean value theorem}
		Let $f\in \mathcal{LP}[a,b]\cap\mathcal{D}[a,b]$, where $\mathcal{D}[a,b]$ is the set of all functions having intermediate value property. Let $g\in \mathcal{LP}[a,b]$. If $g$ is non-negative and $fg\in \mathcal{LP}[a,b]$, then
		\[
		\int_{a}^{b}fg=f(\xi)\int_{a}^{b}g\quad\text{for some $\xi\in[a,b]$}.
		\]
		In particular, we have $\int_{a}^{b}f=f(\xi)(b-a)$ for some $\xi\in[a,b]$.
	\end{theorem}
	\begin{proof}
		We only prove the following case,
		\begin{align*}
		-\infty<\inf\limits_{x\in[a,b]}f(x)<f(x)<\sup\limits_{x\in[a,b]}f(x)=\infty,\,\,\forall x\in[a,b].
		\end{align*}
		Proofs of other possible cases are similar. Let $A=\left\{ x\in[a,b]\mid g(x)>0 \right\}$. If $\int_{a}^{b}g=0$, then $\lambda(A)=0$ and in this case, $\int_{a}^{b}fg=0=f(a)\int_{a}^{b}g$. Assume $\lambda(A)>0$ and $m=\inf\limits_{x\in[a,b]}f(x)$. Then $(f-m)g>0$ on $A$, $\int_{a}^{b}(f-m)g>0$ and $\int_{a}^{b}g>0$; which implies $m<\int_{a}^{b}fg/\int_{a}^{b}g<\infty$. So there is $\kappa,\eta\in[a,b]$ such that $f(\kappa)<\int_{a}^{b}fg/\int_{a}^{b}g<f(\eta)$. This completes the proof.     	
	\end{proof}

	\begin{theorem}[\bf The second mean value theorem]\label{2nd mean}
		Let $f\in \mathcal{LP}[a,b]$ and let $g\in \mathcal{BV}[a,b]$ be such that $g$ does not change sign. If $G$ is a primitive of $g$, then
		\begin{gather*}
		\int_{a}^{b}fG=G(a)\int_{a}^{\xi}f+G(b)\int_{\xi}^{b}f\quad\text{for some $\xi\in[a,b]$}.
		\end{gather*}
	\end{theorem}
	
	\par The proof of this theorem is similar to that of Theorem 12.5 of \cite[p.~193]{AMTOIntegration}.

	\section{Taylor's theorem}\label{Taylor-theorem}
	
	\par In the following theorem, Laplace integral is used to write the integral remainder form of Taylor's theorem, and it is an extension of the Theorem 1 of \cite{ERTUHK}. Fortunately, the estimates, given in Theorem 4 of \cite{ERTUHK} also work for the remainder given in \eqref{Taylor's remainder}. So instead of giving all of those estimates, we only give the sup norm estimate of the remainder. For another form of the remainder, see Theorem 1.6 of \cite{TLD2}.
	
	\begin{theorem}[\bf Taylor's theorem]\label{Taylor's theorem}
		Let $f$, $f^{'}$,...,$f^{(n)}$ be continuous on $[a,b]$ and let $LD_{n+1}f$ exists a.e. on $[a,b]$. If
		\begin{align}\label{Taylor0}
		f^{(n)}(x)-f^{(n)}(a)=\int_{a}^{x}LD_{n+1}f,
		\end{align}
		then
		\begin{align}
		f(x)=\sum\limits_{k=0}^{n}\frac{f^{(k)}(a)}{k!}(x-a)^{k}+R_{n,a}(x)\quad\text{for all $x\in [a,b]$}, \label{Taylor1}
		\end{align}
		where
		\begin{align}\label{Taylor's remainder}
		R_{n,a}(x)=\frac{1}{n!}\int_{a}^{x}LD_{n+1}f(t)(x-t)^{n}\,dt.
		\end{align}
		Also,
		\begin{align}\label{Taylor's remainder estimate}
		\|R_{n,a}\|_{\infty}\leqslant\frac{(b-a)^{n}}{n!}\|LD_{n+1}f\|_{[a,b]}. 
		\end{align}
	\end{theorem}
	\begin{proof}
		We will prove this theorem by induction. For $n=0$, \eqref{Taylor0} and\eqref{Taylor1} are equivalent. Assume \eqref{Taylor1} is true for $n=k$ and \eqref{Taylor0} holds for $n=k+1$. Then by Theorem \ref{Continuity implies differentiability} and Theorem \ref{IntByParts1}, we get
		\begin{align*}
		f(x)&-\sum\limits_{i=0}^{k}\frac{f^{(i)}(a)}{i!}(x-a)^{i}=R_{k,a}(x)\\
		&=\frac{1}{k!}\int_{a}^{x}LD_{k+1}f(t)(x-t)^{k}\,dt=\frac{1}{k!}\int_{a}^{x}f^{(k+1)}(t)(x-t)^{k}\,dt\\
		&=\frac{f^{(k+1)}(a)}{(k+1)!}(x-a)^{k+1}+\frac{1}{(k+1)!}\int_{a}^{x}LD_{k+2}f(t)(x-t)^{k+1}\,dt\\
		&=\frac{f^{(k+1)}(a)}{(k+1)!}(x-a)^{k+1}+R_{k+1,a}(x).
		\end{align*}
		Thus
		\[
		f(x)=\sum\limits_{k=0}^{n}\frac{f^{(k)}(a)}{k!}(x-a)^{k}+R_{n,a}(x).
		\]
		Again by Theorem \ref{2nd mean}, we get
		\begin{align*}
		|R_{n,a}(x)|&=\left| \frac{(x-a)^{n}}{n!}\int_{a}^{\xi}LD_{n+1}f(t) \right|\quad\text{for some $\xi\in[a,x]$}\\
		&\leqslant \frac{(x-a)^{n}}{n!}\|LD_{n+1}f\|_{[a,x]}
		\end{align*}
		and it follows
		\[\|R_{n,a}\|_{\infty}\leqslant\frac{(b-a)^{n}}{n!}\|LD_{n+1}f\|_{[a,b]}.\qedhere\]
	\end{proof}

	\begin{example}
		Let $f$ be the function given in \eqref{Eq13} and let $F_{n}$ be the nth primitive of $f$, i.e., $F_{n}^{(n)}=f$ on $[0,1]$. Then
		\begin{align}\label{Taylor's function}
		F_{n}(x)=\sum\limits_{k=0}^{n}\frac{F_{n}^{(k)}(0)}{k!}x^{k}+\int_{0}^{x}LD_{n+1}F_{n}(t)\,dt\qquad\text{for $x\in [0,1]$}.
		\end{align}
		The integral sign in \eqref{Taylor's function} cannot be replaced by $(\mathsf{P})\int_{0}^{x}$, as $LD_{n+1}F_{n}=LD_{1}f$ is not Perron integrable on $[0,1]$ (Theorem \ref{remarkable theorem}).
	\end{example}

	\section{Poisson integral}\label{Poisson integral}

	\par Let us first define some mathematical conventions.
	We denote the partial differential operator of order $k\in \mathbb{N}$ with respect to $x$ by $\partial_{x}^{k}$. And if $F$ is function defined over $\mathbb{R}^{n}$, the we denote $\partial_{x}^{0}F(x_{1},...,x_{n})=F(x_{1},...,x_{n})$. We also denote the total variation of a function $g(x_{1},...,x_{n})$ with respect to the variable $x_{i}$ on an interval $I$ by $V_{I}[g_{n}(x_{1},...,x_{i-1},\,\cdot\,,x_{i+1},...,x_{n})]$.
	
	\begin{lemma}\label{Poisson lemma}
		Let $I_{0}=I_{1}\times I_{2}\times I_{3}$, where $I_{1}$, $I_{2}$ be two intervals in $\mathbb{R}$ and $I_{3}$ is a compact interval in $\mathbb{R}$. Let  $g(x,y,t)$ be a real-valued function defined over $I_{0}$ such that $\partial_{t}^{1}g(x,y,t)$ is continuous on $I_{0}$. If $\{h_{n}\}$ is a sequence  converging to $0$ such that for all $n$, $x+h_{n}\in I_{1}$, and if we define $g_{n}(x,y,t)=g(x+h_{n},y,t)$, then for each $(x,y)\in I_{1}\times I_{2}$, the sequence of reals $\{V_{I_{3}}[g_{n}(x,y,\cdot)]\}$ is bounded.
	\end{lemma}
	\begin{proof}
		Fix $(x,y)\in I_{1}\times I_{2}$ and choose a compact interval, $J_{x}$, of $x$ such that $J_{x}\subseteq I_{1}$ and $x+h_{n}\in J_{x}$ for all $n$. Then,
		\begin{align*}
		\left|\partial_{t}^{1}g_{n}(x,y,t)\right|=\left|\partial_{t}^{1}g(x+h_{n},y,t)\right|\leqslant \sup\limits_{(u,t)\in J_{x}\times I_{3}}\left|\partial_{t}^{1}g(u,y,t)\right|=M(y),
		\end{align*}
		where $M$ is a real-valued function on $I_{2}$. Which gives
		\[
		V_{I_{3}}[g_{n}(x,y,\cdot)]=\int\limits_{I_{3}}\left|\partial_{t}^{1}g_{n}(x,y,t)\right|\,dt\leqslant \lambda(I_{3})M(y).
		\]
		Thus $\{V_{I_{3}}[g_{n}(x,y,\cdot)]\}$ is bounded.
	\end{proof}

	\par Let $U$ be the open unit disc centred at the origin of $\mathbb{R}^{2}$, let $\partial U$ be the boundary of $U$ and let $f\colon \partial U\to \mathbb{R}$. If $G_{f}(t)=f(e^{it})$ is Lebesgue integrable on $[-\pi, \pi]$, then the Poisson integral of $f$ on $U$ is defined by $F(r,\theta)=(1/2\pi)\int_{-\pi}^{\pi}G_{f}(t)P_{r}(\theta-t)\,dt$, where $r\in[0,1)$ and $\theta\in[-\pi,\pi]$ and $P_{r}(\theta-t)=(1-r^{2})/(1-2r\cos(\theta-t)+r^{2})$ (Section 11.6 of \cite[p.~233]{RealComplexAnalysis}). For $G_{f}\in \mathcal{LP}[-\pi, \pi]$, it is also possible to define $F(r,\theta)$, which is a harmonic function on $U$.
	
	\begin{theorem}\label{Poisson theorem}
		Let $U$ be the open unit disc centred at the origin of $\mathbb{R}^{2}$, let $f\colon \partial U\to \mathbb{R}$ and let $G_{f}\in \mathcal{LP}[-\pi, \pi]$. Then 
		\begin{align}\label{Poisson function}
		F(r,\theta)=\frac{1}{2\pi}\int_{-\pi}^{\pi}G_{f}(t)P_{r}(\theta-t)\,dt
		\end{align}
		is well-defined, where $(r,\theta)\in [0,1)\times[-\pi,\pi]$. Moreover, $F(r,\theta)$ is harmonic function on $U$.
	\end{theorem}
	\begin{proof}
		Let us denote $P(r,\theta,t)=P_{r}(\theta-t)$ and $I_{0}=[0,1)\times[-\pi, \pi]\times[-\pi, \pi]$. Then it is obvious that for any non-negative integers $k$, $l$ and $m$; $\partial_{t}^{k}\partial_{r}^{l}\partial_{\theta}^{m}P$ exists and continuous on $I_{0}$. Thus the continuity of $\partial_{t}^{2}\partial_{r}^{k}P$ implies 
		$$V_{[-\pi, \pi]}[\partial_{t}^{1}\partial_{r}^{k}P(r,\theta,\cdot)]<\infty.$$
		So $(1/2\pi)\int_{-\pi}^{\pi}G_{f}(t)\partial_{r}^{k}P(r,\theta,t)\,dt$ is well-defined for all non-negative integer $k$. We now prove that for any positive integer $k$,
		\begin{align}\label{Poisson function1}
		\partial_{r}^{k}F(r,\theta)=\frac{1}{2\pi}\int_{-\pi}^{\pi}G_{f}(t)\partial_{r}^{k}P(r,\theta,t)\,dt.
		\end{align}
		Accordingly, if we can prove \eqref{Poisson function1} for $k=1$, then it will follow by straightforward induction. Let $\{h_{n}\}$ be any sequence converging to $0$ such that $r+h_{n}\in [0,1)$ and $h_{n}\neq 0$ for all $n$. Then
		\begin{align*}
		\frac{F(r+h_{n},\theta)-F(r,\theta)}{h_{n}}&=\frac{1}{2\pi}\int_{-\pi}^{\pi}G_{f}(t)\left(\frac{P(r+h_{n},\theta,t)-P(r,\theta,t)}{h_{n}}\right)\,dt\\
		&=\frac{1}{2\pi}\int_{-\pi}^{\pi}G_{f}(t)\partial_{r}^{1}P(r+\xi_{n},\theta,t)\,dt\qquad\text{for $0<\xi_{n}<h_{n}$}\\
		&=\frac{1}{2\pi}\int_{-\pi}^{\pi}G_{f}(t)g_{n}(r,\theta,t)\,dt,
		\end{align*}
		where $g_{n}(r,\theta,t)=\partial_{r}^{1}P(r+\xi_{n},\theta,t)$. As 
		$$\partial_{t}^{2}g_{n}(r,\theta,t)=\partial_{t}^{2}\partial_{r}^{1}P(r+\xi_{n},\theta,t)$$
		is continuous on $I_{0}$, by Lemma \ref{Poisson lemma}, we get $\{V_{[-\pi, \pi]}[\partial_{t}^{1}g_n(r,\theta,.)]\}$ is bounded. Thus by Theorem \ref{sequence of integrals}, we get
		\begin{gather*}
		\partial_{r}^{1}F(r,\theta)=\lim\limits_{n\rightarrow \infty}\frac{1}{2\pi}\int_{-\pi}^{\pi}G_{f}(t)g_{n}(r,\theta,t)\,dt=\frac{1}{2\pi}\int_{-\pi}^{\pi}G_{f}(t)\partial_{r}^{1}P(r,\theta,t)\,dt.
		\end{gather*}
		Proof of $\partial_{\theta}^{k}F(r,\theta)=(1/2\pi)\int_{-\pi}^{\pi}G_{f}(t)\partial_{\theta}^{k}P(r,\theta,t)\,dt$ for $k\in\mathbb{N}$, is similar. If we denote $\Delta=\partial_{r}^{2}+\frac{1}{r}\partial_{r}^{1}+\frac{1}{r^{2}}\partial_{\theta}^{2}$, which is the polar form of the Laplacian, then we get
		\[
		\Delta F(r,\theta)=\frac{1}{2\pi}\int_{-\pi}^{\pi}G_{f}(t)\Delta P_{r}(\theta-t)\,dt=0,
		\]
		which completes the proof.
	\end{proof}

	\begin{theorem}[\bf Boundary behaviour of Poisson integrals]
		Let $U$ be the open unit disc centred at the origin of $\mathbb{R}^{2}$, let $f\colon \partial U\to \mathbb{R}$ and let $G_{f}\in \mathcal{LP}[-\pi, \pi]$. If we denote $F_{r}(\theta)=F(r,\theta)$ {\upshape(see \eqref{Poisson function})}, where $0\leqslant r<1$ and $-\pi\leqslant \theta \leqslant \pi$, then
		\begin{align}\label{BoundaryCondition-1}
		\|F_{r}\|_{[-\pi,\pi]}\leqslant \|G_{f}\|_{[-\pi,\pi]}.
		\end{align}
		Moreover,
		\begin{align}\label{BoundaryCondition-2}
		\|F_{r}-G_{f}\|_{[-\pi,\pi]}\rightarrow 0\quad\text{as \quad$r\rightarrow 1$.}
		\end{align}
	\end{theorem}
	\begin{proof}
		As $P_{r}(\theta-t)>0$ and $P_{r}(\theta-\pi)=P_{r}(\theta+\pi)$, by Theorem \ref{2nd mean}, we get
		\begin{align*}
		\left|2\pi F_{r}(\theta)\right|&= \left|P_{r}(\theta+\pi)\int_{-\pi}^{\xi}G_{f} + P_{r}(\theta-\pi)\int_{\xi}^{\pi}G_{f}\right|\\
		&=\left|\int_{-\pi}^{\pi}G_{f}\right|P_{r}(\theta+\pi)\leqslant \|G_{f}\|_{[-\pi,\pi]}P_{r}(\theta+\pi),
		\end{align*}
		which gives, $\left|\int_{-\pi}^{x}F_{r}(\theta)\,d\theta\right|\leqslant (1/2\pi)\|G_{f}\|_{[-\pi,\pi]}\int_{-\pi}^{\pi}P_{r}(\theta+\pi)\,d\theta=\|G_{f}\|_{[-\pi,\pi]}$ for $x\in [-\pi,\pi]$. Thus
		\[
		\|F_{r}\|_{[-\pi,\pi]}\leqslant \|G_{f}\|_{[-\pi,\pi]}.
		\]
		\par To prove \eqref{BoundaryCondition-2}, choose an $\epsilon\,(>0)$. By Theorem \ref{Denseness of plynomials}, there exists a polynomial $q\colon\partial U\to \mathbb{R}$ such that $\|G_{q}-G_{f}\|_{[-\pi,\pi]}<\epsilon$. Let $Q(r,\theta)$ be the Poisson integral of $q$ and let $Q_{r}(\theta)=Q(r,\theta)$. Then
		\[
		F_{r}-G_{f}=\left( F_{r}-Q_{r} \right) + \left( Q_{r}-G_{q} \right) + \left( G_{q}-G_{f}\right).
		\]
		Now \eqref{BoundaryCondition-1} gives, $\|F_{r}-Q_{r}\|_{[-\pi,\pi]}\leqslant \|G_{f}-G_{q}\|_{[-\pi,\pi]}<\epsilon$. Furthermore, note that, if $h\in \mathcal{LP}[a,b]$, then
		\begin{align*}
		\|h\|_{[-\pi,\pi]}=\sup\limits_{x\in [a,b]}\left| \int_{a}^{x}h \right|\leqslant \int_{a}^{b}\left|h\right|=\|h\|_{1},
		\end{align*}
		where $\|h\|_{1}$ is the well known $L_{1}$-norm. Which implies $\|Q_{r}-G_{q}\|_{[-\pi,\pi]}\leqslant \|Q_{r}-G_{q}\|_{1}$. Thus
		$$\|F_{r}-G_{f}\|_{[-\pi,\pi]}\leqslant 2\epsilon + \|Q_{r}-G_{q}\|_{1}.$$
		Now as $\|Q_{r}-G_{q}\|_{1}\rightarrow 0$ for $r\rightarrow 1$ (Theorem 11.16 of \cite[p.~239]{RealComplexAnalysis}), and $\epsilon$ is arbitrary, we get $\|F_{r}-G_{f}\|_{[-\pi,\pi]}\rightarrow 0$ for $r\rightarrow 1$.
	\end{proof}

	\section{Existence and uniqueness theorems for a system of generalised ordinary differential equations and higher-order generalised ordinary differential equation}\label{Exist-uniqu-IVP}
	
	\par Here we give a local existence and uniqueness theorem for a system of generalised ordinary differential equations and then with the help of that a local existence and uniqueness theorem for higher-order generalised ordinary differential equation will be given.
	\par Consider the following system of generalised ordinary differential equations with an initial condition.
	\begin{align}\label{SystemGeneralisedIVP}
	\begin{cases}
	LD_{1}x_{1}(t)=f_{1}(t,x_{1}(t),...,x_{n}(t))\\
	LD_{1}x_{2}(t)=f_{2}(t,x_{1}(t),...,x_{n}(t))\\
	\qquad\qquad\qquad...\\
	LD_{1}x_{n}(t)=f_{n}(t,x_{1}(t),...,x_{n}(t));
	\end{cases}
	\end{align}
	where $t\in I$ and $x_{i}(t_{0})=\alpha_{i}$ for some $t_{0}\in I$, $1\leqslant i\leqslant n$. If we write ${\bf x}=(x_{1},...,x_{n})$, $\overline{\bf x}={\bf x}^{T}$, $f_{i}(t,{\bf x})=f_{i}(t,x_{1},...,x_{n})$ for $i=1,...,n$, $\overline{f}(t,\overline{\bf x})=(f_{1}(t,{\bf x}),...,f_{n}(t,{\bf x}))^{T}$ and $LD_{1}\overline{\bf x}(t)=(LD_{1}x_{1}(t),...,LD_{1}x_{n}(t))^{T}$ then \eqref{SystemGeneralisedIVP} can be written as follows:
	\begin{align}\label{EqSystemGeneralisedIVP}
	\begin{cases}
	LD_{1}\overline{\bf x}(t)=\overline{f}(t,\overline{\bf x}(t)),\\
	\overline{\bf x}(t_{0})=(\alpha_{1},...\alpha_{n})^{T}
	\end{cases}
	\end{align}
	where $t,t_{0}\in I$. We also write $\int\overline{f}(t,\overline{\bf x})= (\int f_{1}(t,{\bf x}),...,\int f_{n}(t,{\bf x}))^{T}$ and say $\overline{f}(t,\overline{\bf x}(t))$ is Laplace continuous at $t=t^{*}\in I$ if for each $i=1,...,n$, $f_{i}(t,{\bf x}(t))$ is Laplace continuous at $t=t^{*}\in I$. Without violating the generality of \eqref{EqSystemGeneralisedIVP} we can assume $t_{0}$ to be an interior point of $I$.
	
	\begin{theorem}\label{ThSystemGeneralisedIVP}
		Let $I$ be an interval and let $t_{0}$ be an interior point of $I$. If
		\begin{enumerate}[\upshape A.]
			\item $\overline{f}(t,\overline{\bf x}(t))$ is Laplace continuous on $I$ for all $\overline{\bf x}\in \mathcal{BLC}^{n}(I)$ and
			\item there is a neighbourhood $U$ of $t_{0}$ in $I$ such that for each $i=1,...,n$ and $\overline{\bf x}, \overline{\bf y}\in \mathcal{BLC}^{n}(I)$,
			\[
			\left|f_{i}(t,{\bf x}(t))-f_{i}(t,{\bf y}(t))\right|\leqslant v_{t_{0}}(t)\|\overline{\bf x}-\overline{\bf y}\|_{n,\infty}\qquad\text{a.e. on $U$,}
			\]
			where $v_{t_{0}}$ is a Lebesgue integrable function on $U$,
		\end{enumerate}
		then there exists a unique solution of \eqref{EqSystemGeneralisedIVP} on $[t_{0}-a, t_{0}+a]\subseteq U$, for some $a(>0)$.
	\end{theorem}
	
	\begin{proof}
		Choose an $a(>0)$ such that
		\begin{align}\label{Lipschitz-function}
		\left|\int_{t_{0}}^{t}v_{t_{0}}(s)\,ds\right|\leqslant \frac{1}{2}\qquad\text{for $t\in [t_{0}-a, t_{0}+a]\subseteq I$.}
		\end{align}
		Let us denote $\mathcal{BLC}^{n}([t_{0}-a, t_{0}+a])$ by $X$ and $(\alpha_{1},...,\alpha_{n})$ by $\alpha$. Then $\left( X, \|\,.\,\|_{n,\infty} \right)$ is complete. Define $F:X\to X$ as follows:
		\[
		F(\overline{\bf x}(t))=\overline{\alpha}+\int_{t_{0}}^{t}\overline{f}(s,\overline{\bf x}(s))\,ds.
		\]
		We claim $F$ is a contraction mapping. Let $\overline{\bf x}, \overline{\bf y}\in X$. Then by the second condition of this theorem we get, for each $i=1,...,n$,
		\begin{align*}
		\left| F(x_{i})(t)-F(y_{i})(t) \right|&\leqslant \left|\int_{t_{0}}^{t}f_{i}(s,{\bf x}(s))- f_{i}(s,{\bf y}(s))\,ds\right|\\
		&\leqslant \|\overline{\bf x}-\overline{\bf y}\|_{n,\infty}\left|\int_{t_{0}}^{t}v(s)\,ds\right|\leqslant \frac{1}{2}\|\overline{\bf x}-\overline{\bf y}\|_{n,\infty};
		\end{align*}
		which implies $\|F(\overline{\bf x})-F(\overline{\bf y})\|_{n,\infty}\leqslant (1/2)\|\overline{\bf x}-\overline{\bf y}\|_{n,\infty}$ for all $\overline{\bf x}, \overline{\bf y}\in X$; and it proves our claim. Now by the contraction mapping theorem $F$ has exactly one fixed point, $\overline{\bf x}^{*}\in X$. Moreover, $\overline{\bf x}^{*}$ is the uniform limit of the sequence $\overline{\bf x}_{n}=F(\overline{\bf x}_{n-1})$, $n\in \mathbb{N}$; where $\overline{\bf x}_{0}=\overline{\bf x}(t_{0})$ on $[t_{0}-a, t_{0}+a]$. Thus
		\[
		\overline{\bf x}^{*}(t)=\overline{\alpha}+\int_{t_{0}}^{t}\overline{f}(s,\overline{\bf x}^{*}(s))\,ds.
		\]
		Now by the first condition of this theorem we get, $\overline{\bf x}^{*}$ is the unique element of $X$ such that
		\begin{align*}
		\begin{cases}
		LD_{1}\overline{\bf x}^{*}(t)=\overline{f}(t,\overline{\bf x}^{*}(t))\\
		\overline{\bf x}^{*}(t_{0})=(\alpha_{1},...\alpha_{n})^{T}.
		\end{cases}
		\end{align*}
		for $t\in I$.
	\end{proof}
	 
	\begin{theorem}
		Under the assumption of {\upshape Theorem \ref{ThSystemGeneralisedIVP}}, the initial value problem
		\begin{align}\label{HigherOrderDifferentialEquation}
		\begin{cases}
		LD_{n}x(t)=f(t,x(t),x^{'}(t),...,x^{(n-1)}(t)),\qquad \left( x^{(k)}=\dfrac{d^{k}x}{dt^{k}}\right)\\
		x(t_{0})=\alpha_{1}, x^{'}(t_{0})=\alpha_{2},...,x^{(n-1)}(t_{0})=\alpha_{n}
		\end{cases}
		\end{align}
		has an unique solution on $[t_{0}-a, t_{0}+a]\subseteq U$, for some $a(>0)$.
	\end{theorem}
	\begin{proof}
		Using Theorem \ref{THEOREm}, Theorem 10 of \cite{OLD} and the first condition of this theorem, it can be proved that \eqref{HigherOrderDifferentialEquation} is equivalent to the following system of generalised ordinary differential equations:
		\begin{align*}
		\begin{cases}
		x_{1}^{'}(t)= x_{2}(t)\\
		x_{2}^{'}(t)= x_{3}(t)\\
		\qquad\,\,...\\
		x_{n-1}^{'}(t)= x_{n}(t)\\
		LD_{1}x_{n}(t)= f_{n}(t,x_{1}(t),...,x_{n}(t));
		\end{cases}
		\end{align*}
		where $t\in I$, $x_{i}(t_{0})=\alpha_{i}$ for $i=1,...,n$, $x(t)=x_{1}(t)$ and $f_{n}=f$. By the second condition it is evident that
		\[
		\overline{f}(t,\overline{\bf x}(t))=(x_{2}(t),x_{3}(t),..,x_{n}(t),f_{n}(t,{\bf x}(t)))^{T}
		\]
		satisfies the conditions of Theorem \ref{ThSystemGeneralisedIVP} and this completes the proof.
	\end{proof}

	\noindent\textbf{Acknowledgements:} This research work of the first author is supported by the UGC fellowship of India (Serial no.- 2061641179, Ref. no.- 19/06/2016(i) EU-V and Roll no.- 424175 under the UGC scheme). UGC's financial support is highly appreciated.
	
	\bibliographystyle{plain} 
	\bibliography{MyBibliography.bib}
	
\end{document}